\documentclass[hidelinks,11pt]{article}

\usepackage{amsmath}
\usepackage{amssymb}
\usepackage{bbm}
\usepackage{amsfonts,amsthm}
\usepackage{thmtools,thm-restate}
\usepackage{tabularx,lipsum,environ,float}
\usepackage{multirow}
\usepackage{hyperref}
\usepackage{graphicx}

\newtheorem{theorem}{Theorem}[section]
\newtheorem{lemma}[theorem]{Lemma}
\newtheorem{definition}[theorem]{Definition}

\newtheorem{claim}[theorem]{Claim}

\newtheorem{proposition}[theorem]{Proposition}
\newtheorem{observation}[theorem]{Observation}

\theoremstyle{definition}

\newcommand\card[1]{|#1|}
\newcommand\ind{\mathbbm{1}}

\newenvironment{claimproof}{\begin{proof}\renewcommand{\qedsymbol}{\claimqed}}{\end{proof}\renewcommand{\qedsymbol}{\plainqed}}
\let\plainqed\qedsymbol

\usepackage{authblk}

\usepackage{subcaption}
\usepackage[a4paper,bindingoffset=0.2in,
            top=1.1in, bottom=1.2in, left=1.1in, right=1.1in,
            footskip=.5in]{geometry}

\usepackage{tikz}
\usetikzlibrary{graphs,graphs.standard}
\usetikzlibrary{calc}
\usetikzlibrary{decorations.pathreplacing}
\usetikzlibrary{backgrounds}

\tikzset{
	auto,
	node distance = 0.5cm,
	vtx/.style={
		circle,
		fill=black,
		draw,
		thick,
		minimum size=.08cm,
		inner sep=0pt
	},
	edge/.style = {draw},
	coveredge/.style = {draw, thick, color=white},
	etcedge/.style = {draw, dotted},
	nonedge/.style = {draw, dotted},
	diredge/.style = { ->},
}

\makeatletter
\newcommand{\gettikzxy}[3]{%
  \tikz@scan@one@point\pgfutil@firstofone#1\relax
  \edef#2{\the\pgf@x}%
  \edef#3{\the\pgf@y}%
}
\makeatother

\newcommand{\planar}{{\mathcal{M}_{planar}}}
\newcommand{\outerplanar}{{\mathcal{M}_{outerplanar}}}
\newcommand{\class}{{\mathcal{M}_{\mathcal{C}}}}

\usepackage{stars}

\title{On the maximum number of edges in planar graphs of bounded degree and matching number}

\author[1]{Lars Jaffke\footnote{Supported by the Norwegian Research Council (NFR, project 274526).}}
\author[2]{Paloma T. Lima} 

\affil[1]{Department of Informatics, University of Bergen, Norway}
\affil[ ]{\texttt{lars.jaffke@uib.no}}

\affil[2]{Computer Science Department, IT University of Copenhagen, Denmark}
\affil[ ]{\texttt{palt@itu.dk}}

\date{}

\begin{document}

\maketitle

\begin{abstract}
We determine the maximum number of edges that a planar graph can have as a function of its maximum degree and matching number. \\

\noindent{\small \texttt{Keywords: planar graphs, matching number, extremal problems}}

\end{abstract}

\section{Introduction}

A typical problem in combinatorics is to determine the maximum number of edges that a graph can have under some structural constraint. For instance, the widely investigated Tur\'{a}n numbers are concerned with the maximum number of edges that a graph can have as a function of its number of vertices, if it does not have a fixed graph $H$ as an induced subgraph. In this work, we are interested in understanding how large a graph can be in terms of its degree and matching number. That is, we want to determine the maximum number of edges that a graph can have if its maximum degree and matching number are both bounded. In contrast with the previously mentioned Tur\'{a}n numbers, we do not impose any constraints on the number of vertices. 
Therefore, in general, if one of the two parameters is unbounded, there is no upper bound on the number of edges that a graph can have: 
a star can have unbounded number of edges while having matching number one, and 
a disjoint union of single edges can also have an unbounded number of edges while having maximum degree one. 
By Vizing's Theorem, every graph can have its edge set partitioned into a family of at most $\Delta(G)+1$ matchings, where $\Delta(G)$ denotes the maximum degree of the graph $G$. Thus, bounding both the maximum degree and the matching number is actually enough to bound the number of edges that a graph can have. The question here, which dates back to 1960, is how tight of a bound this is.

Chv\'atal and Hanson~\cite{CHVATAL76} showed an even better (and tight) upper bound on this value, in the case where no further restrictions are imposed on the graphs considered. Later on, Balachandran and Khare~\cite{BAL09} gave a constructive proof of the same result, which made it possible to identify the structure of the graphs achieving the given bound on the number of edges. Such graphs are called edge-extremal graphs. Following these studies, the number of edges of edge-extremal graphs were also investigated when additional structural constraints are imposed on the graphs. For instance, claw-free graphs were studied by Dibek et al.~\cite{DIBEK17}, bipartite graphs, split graphs, disjoint unions of split graphs and unit interval graphs by M\aa land~\cite{MALAND15} and, very recently, chordal graphs were considered by Blair et al.~\cite{LATIN2020}. In all these works, the authors are able to understand how additional structure affects the maximum number of edges that a graph can have: they provide tight upper bounds and provide examples of graphs achieving these bounds.

In this work, we turn our attention to the well-known class of planar graphs, which has already been the object of study for many other combinatorial questions of similar flavour~\cite{ThomassenJGT,DOWDEN,GHOSH2021,Planar3-path,HakimiJGT,EJC}. We determine the maximum number of edges that a planar graph can have, given the constraints on its maximum degree and matching number. Given $d,\nu\in\mathbb{N}$, we denote by $\planar(d,\nu)$ the set of planar graphs such that
$\Delta(G)<d$ and $\nu(G)<\nu$. (We require strict inequality in this definition to be consistent with the literature.) 
Our main result is summarized in the following theorem.

\begin{theorem}\label{thm:main}
	Let $d, \nu \ge 2$ be two integers.
	If $G$ is a graph in $\planar(d, \nu)$, 
	then
	\begin{align*}
		\card{E_G} \le 
		\begin{cases}
			(d-1)(\nu-1), &\mbox{ if } d = 2 \mbox{ or } d \ge 7 \\
			d(\nu-1), &\mbox{ if } d = 3 \\
			(d-1)(\nu-1) + \left\lfloor\frac{\nu-1}{2}\right\rfloor, &\mbox{ if } d \in \{4, 5\} \\
			(d-1)(\nu-1) + 2\cdot\left\lfloor\frac{\nu-1}{7}\right\rfloor + \ind[\nu-1 \bmod 7 \ge 4], &\mbox{ if } d = 6 \\
		\end{cases}
	\end{align*}
	where $\ind[\psi]$ equals $1$ if $\psi$ is true and $0$ otherwise.
	In all of the above cases, the bounds are tight.
\end{theorem}

For each of the above cases of Theorem~\ref{thm:main}, we construct concrete examples of graphs achieving the maximum number of edges. This paper is organized as follows. In Section~\ref{sec:prelim} we define the necessary concepts and state preliminary results that will be used throughout the proof. Section~\ref{sec:proof} is entirely devoted to the proof of Theorem~\ref{thm:main}. We conclude the paper in Section~\ref{sec:conclusion} with problems for further research.

\section{Preliminaries}\label{sec:prelim}

Given two integers $a$ and $p$, we denote by $(a \bmod p)$ the remainder after dividing $a$ by $p$. In particular, $(a \bmod p)$ is an integer in $\{0,\ldots,p-1\}$.

All the graphs considered in this paper are simple and undirected. We denote by $V_G$ and $E_G$ the vertex set and edge set of $G$, respectively. Given $v\in V_G$, we denote by $N_G(v)$ the set of vertices that are adjacent to~$v$. The \emph{degree} of $v$ is denoted by $\deg_G(v)$ and is defined as $|N_G(v)|$. The \emph{degree of a graph $G$} is the maximum degree of a vertex in $G$ and it is denoted by $\Delta(G)$. The \emph{degree sequence} of a graph $G$ with $V_G=\{v_1,v_2,\ldots, v_n\}$ is an integer sequence $\sigma=d_1d_2\ldots d_n$ such that $d_i=\deg_G(v_i)$. It is usual to assume that $d_1\geq d_2\geq\ldots\geq d_n$. We say a graph $G$ \emph{realizes} the degree sequence $\sigma$ if its degree sequence is $\sigma$.

A set $M\subseteq E_G$ is a \emph{matching} if no two edges in $M$ share a common vertex and $M$ is a \emph{perfect matching} if $M$ is a matching and every vertex of $V_G$ is the endpoint of an edge in $M$. The \emph{matching number of $G$}, denoted by $\nu(G)$, is the largest size of a matching in $G$. A graph $G$ is a \emph{factor-critical graph} if for every $v\in V_G$, $G\setminus v$ has a perfect matching, where $G\setminus v$ denotes the graph obtained from $G$ by the removal of $v$. 

A \emph{tree} is a connected acyclic graph. A \emph{star} is a tree with at most one vertex that is not a leaf, and for $k\in \mathbb{N}$, a \emph{$k$-star}, denoted by $K_{1,k}$, is a star with $k$ leaves. A graph is a \emph{complete graph} on $n$ vertices, denoted by $K_n$, if there is an edge between every pair of its vertices. Given two graphs $G$ and $H$, the \emph{disjoint union of $G$ and $H$}, denoted by $G+H$ is the graph with vertex set $V_G\cup V_H$ and edge set $E_G\cup E_H$. We denote by $pH$ the graph that is the disjoint union of $p$ copies of a graph $H$. 

A graph is \emph{planar} if it admits an embedding on the plane without crossing edges. A well-known fact about planar graphs is that its number of edges can be upper bounded in terms of its number of vertices. We state this in the following observation for future reference.

\begin{observation}\label{lem:planar:edges}
	If $G$ is a planar graph on $n\geq 3$ vertices, then $\card{E_G} \le 3n - 6$.
\end{observation}

We will also use the following result concerning degree sequences of planar graphs of maximum degree five. 

\begin{theorem}[\cite{degreeseq}]\label{thm:degreeseq}
There is no planar graph realizing the degree sequence $5^{10}4$. The same holds for $5^{12}4$.
\end{theorem}

Given two integers $d$ and $\nu$ and a graph class $\mathcal{C}$, we denote by $\mathcal{M}_{\mathcal{C}}(d,\nu)$ the set of all graphs $G$ in $\mathcal{C}$ such that $\Delta(G)<d$ and $\nu(G)<\nu$. A graph in $\mathcal{M}_{\mathcal{C}}(d,\nu)$ that has the maximum number of edges is called an \emph{edge-extremal graph}. When the graph class considered is the class of all graphs, we write simply $\mathcal{M}(d,\nu)$.

\begin{lemma}[\cite{BAL09}] \label{lem:factorc}
Let $\mathcal{C}$ be a graph class that is closed under vertex deletion and closed under taking disjoint union with stars. Let $G$ be an edge-extremal graph in $\mathcal{M}_{\mathcal{C}}(d,\nu)$ with maximum number of connected components that are $(d-1)$-stars. Then every connected component of $G$ that is not a $(d-1)$-star is factor-critical. 
\end{lemma}

The following statement gives a summary of the results obtained by Balachandran and Khare~\cite{BAL09}. 

\begin{theorem}[\cite{BAL09}] \label{theo:generalcase}
The maximum number of edges of a graph in $\mathcal{M}(d,\nu)$ is given by $(d-1)(\nu-1)+\lfloor\frac{d-1}{2}\rfloor \lfloor\frac{\nu-1}{\lceil \frac{d-1}{2}\rceil}\rceil$. Moreover, a graph achieving this number of edges is 

\[
  \begin{cases} 
   rK_{1,d-1}+qK'_d, & \text{if } d\text{ is even} \\

   rK_{1,d-1}+qK_d,   & \text{if } d\text{ is odd},
  \end{cases}
\]

\noindent where $\nu-1=q\lceil \frac{d-1}{2}\rceil+r$, with $r\geq 0$, and $K'_d$ is the graph obtained from $K_d$ by the removal of the edges of a perfect matching and addition of a new vertex adjacent to $d-1$ vertices.
\end{theorem}

An \emph{edge coloring} of a graph is an assignment of colors to its edges such that any two edges that share a common endpoint receive distinct colors. Thus, any edge coloring of a graph $G$ uses at least $\Delta(G)$ many colors. The well-known Theorem of Vizing states that every graph can be edge colored with $\Delta(G)+1$ colors. Deciding whether a graph can be colored with $\Delta(G)$ colors is $\mathsf{NP}$-complete~\cite{HOLYER}. One of the very few known sufficient conditions for a graph not to be $\Delta(G)$-edge colorable is when the graph has, in some sense, too many edges. Note that an edge coloring is a partition of the edge set of the graph into matchings. Hence, if a graph $G$ has strictly more than $\Delta(G)\nu(G)$ edges, it cannot be $\Delta(G)$-edge colored. The following useful observation follows directly from the above discussion.

\begin{observation}\label{obs:edgecol}
Let $\mathcal{C}$ be a graph class such that every $G \in \mathcal{C}$ is $\Delta(G)$-edge colorable. Then an edge-extremal graph in $\class(d,\nu)$ has at most $(d-1)(\nu-1)$ edges.
\end{observation}

Before we consider the case of planar graphs, we observe that in the special case of outerplanar graphs, a well-known subclass of planar graphs, Observation~\ref{obs:edgecol} is enough to obtain tight bounds for our problem. A graph is an \emph{outerplanar graph} if it admits a planar embedding in which all vertices are in the same face. Every outerplanar graph~$G$ with $\Delta(G)\geq 3$ can be edge colored with $\Delta(G)$ many colors~\cite{FIORINI}. Hence, the following result is immediate for this graph class.

\begin{proposition}
If $G$ is a graph in $\outerplanar(d,\nu)$, then

\[ |E_G|\leq
  \begin{cases} 
   3(\nu-1), & \text{if } d=3, \\

   (d-1)(\nu -1),   & \text{if } d=2 \text{ or } d\geq 4.
  \end{cases}
\]

\noindent Moreover, graphs achieving this bound are $(\nu-1)K_d$ when $d\in \{2,3\}$ and $(\nu-1)K_{1,d-1}$ when $d\geq 4$.
\end{proposition}

\section{Proof of Theorem~\ref{thm:main}}\label{sec:proof}

In this section we present a proof of Theorem~\ref{thm:main}. It is easy to see that, when $d\leq 4$, the edge-extremal graphs of $\mathcal{M}(d,\nu)$ (see Theorem~\ref{theo:generalcase}) are planar graphs. In particular, when $d=2$ (resp.\ $d=3$) the edge-extremal graphs are given by a disjoint union of $\nu -1$ edges (resp.\ triangles). When $d=4$, an edge-extremal graph is given by a disjoint union of $\lfloor\frac{\nu-1}{2}\rfloor$ copies of $K'_4$ and, if $\nu$ is even, one copy of $K_{1,3}$. We summarize this in the next observation and, from now on, we assume $d\geq 5$.
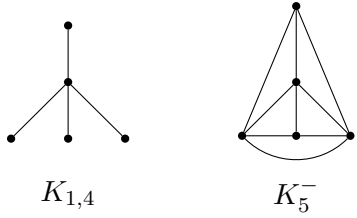
\begin{figure}
	\centering
\begin{tikzpicture}
%
%
	\begin{scope}[xshift=3cm,node distance=0.75cm]
		\fourstar
		\node[] at (0,-1.5) {$K_{1, 4}$};
	\end{scope}
	\begin{scope}[xshift=6cm]
		\node[circle, minimum size=2cm] (C1) at (0,0) {};
		
		\node[vtx] (1) at (0,0) {};
		\node[vtx] (2) at (C1.90) {};
		\node[vtx] (3) at (C1.225) {};
		\node[vtx] (4) at (C1.315) {};
		\node[vtx] (5) at ($(3)!0.5!(4)$) {};
		
		\foreach \i in {2,3,4,5} {
			\draw[edge] (1) -- (\i);
		}
		\foreach \i in {3,4} {
			\draw[edge] (2) -- (\i);
		}
		\foreach \i in {5} {
			\draw[edge] (3) -- (\i);
		}
		\foreach \i in {5} {
			\draw[edge] (4) -- (\i);
		}
		\path[-] (3) edge [bend right=45] (4);
		\node[] at (0,-1.5) {$K_5^-$};
	\end{scope}
\end{tikzpicture}
	\caption{Possible connected components of the pivotal graphs in $\planar(5,\nu)$.}
	\label{fig:pivotal:4}
\end{figure}

\begin{observation}
If $G$ is a graph in $\planar(d,\nu)$ with $d\leq 4$, then

\[ |E_G|\leq
  \begin{cases} 
   \nu-1, & \text{if } d=2, \\

   3(\nu -1),   & \text{if } d=3,\\
   
   7\lfloor\frac{\nu -1}{2}\rfloor+3(\nu-1 \bmod 2), & \text{if } d=4.
  \end{cases}
\]
\end{observation}

By Observation~\ref{obs:edgecol}, together with the fact that every planar graph of maximum degree at least seven can be edge colored with $\Delta(G)$ many colors~\cite{SANDERS2001,Vizing1965}, we obtain the following. 

\begin{observation}\label{obs:lessthan5}
If $G$ is a graph in $\planar(d,\nu)$ with $d\geq 8$, then $|E_G|\leq (d-1)(\nu-1)$. Moreover, equality holds if $G$ is isomorphic to $(\nu-1)K_{1,d-1}$.
\end{observation}

We now focus on the remaining three cases, that is, when $d\in\{5,6,7\}$. In these cases, a substantially more careful analysis is necessary to determine the number of edges in the edge-extremal graphs of $\planar(d,\nu)$.

\newcommand\kfiveminus{\ensuremath{K_5^{-}}}

We start with the case $d=5$. We denote by \kfiveminus~the graph obtained from the complete graph on five vertices by the removal of one edge. Note that $|E_{\kfiveminus} |=9$, $\Delta(\kfiveminus)=4$ and $\nu(\kfiveminus)=2$.

\begin{definition}
For each $\nu$, a graph $G$ is \emph{pivotal in $\planar(5, \nu)$} 
	if $G$ is isomorphic to  
\[
\begin{aligned}
	\frac{\nu-1}{2}\kfiveminus, \quad & \text{if } \nu \text{ is odd},\\
	\left\lfloor\frac{\nu-1}{2}\right\rfloor\kfiveminus+K_{1,4}, \quad & \text{if } \nu \text{ is even}.
\end{aligned}
\]	

\end{definition}

In Figure~\ref{fig:pivotal:4}, we illustrate the possible connected components of the pivotal graphs in $\planar(5,\nu)$.

\begin{lemma}\label{lemma:dis5}
If $G$ is a graph in $\planar(5,\nu)$, then $|E_G|\leq 9\lfloor\frac{\nu-1}{2}\rfloor+ 4 (\nu -1 \bmod 2)$. Moreover, if $G$ is pivotal in $\planar(5,\nu)$, then equality holds.
\end{lemma}

\begin{proof}
Let $G$ be an edge-extremal graph in $\planar(5,\nu)$ with maximum number of connected components that are isomorphic to stars and, subject to that, with maximum number of connected components isomorphic to \kfiveminus. 

\begin{claim}
Every connected component of $G$ is isomorphic to $K_{1,4}$ or to \kfiveminus.
\end{claim}

\begin{claimproof}
Suppose this is not the case. Let $H$ be such a component. If $H$ is a star and it is not isomorphic to $K_{1,4}$, then since $\Delta(H)\leq 4$, $H$ is isomorphic to $K_{1,a}$ with $a\leq 3$. In this case we can replace the connected component $H$ in $G$ by a copy of $K_{1,4}$ and obtain a graph that has degree at most four and the same matching number as $G$, but has strictly more edges than $G$. This is a contradiction with the fact that $G$ is edge-extremal in $\planar(5,\nu)$. Hence $H$ is not a star. Let $\nu_1$ be the size of a maximum matching in $H$. By Lemma~\ref{lem:factorc}, $H$ is factor-critical and therefore $|V_H|=2\nu_1+1$. Since $\Delta(H)\leq 4$, we have that $2|E_H|\leq 4|V_H|=4(2\nu_1+1)$. That is, $|E_H|\leq 4\nu_1+2$. 

If $\nu_1$ is even and $\nu_1\geq 4$, let $H'$ be a pivotal graph in $\planar(5, \nu_1 +1)$. That is, since $\nu_1+1$ is odd, $H'$ is given by the disjoint union of $\frac{\nu_1}{2}$ copies of \kfiveminus. Note that $\nu(H')=\nu_1$ and $|E_{H'}|=9\frac{\nu_1}{2}$.   Since $\nu_1\geq 4$, $|E_{H'}|\geq 4\nu_1+2$. Hence, we can replace the connected component $H$ in $G$ by $H'$, and obtain a graph in $\planar(5,\nu)$ with the same matching number and at least as many edges as $G$. However, this graph has strictly more connected components isomorphic to \kfiveminus, which is a contradiction with our initial choice of $G$.

Similarly, if $\nu_1$ is odd and $\nu_1\geq 5$, let $H'$ be a pivotal graph in $\planar(5, \nu_1 +1)$. Then $\nu(H')=\nu_1$, $|E_{H'}|=9\frac{\nu_1-1}{2}+4$, and since $\nu_1\geq 4$, $|E_{H'}|\geq 4\nu_1+2$. By replacing $H$ by $H'$ in $G$, we again reach a contradiction with the maximality of the number of connected components of $G$ that are isomorphic to stars and to \kfiveminus. It remains to consider the cases $\nu_1\in\{2,3\}$.

 If $\nu_1=2$, then $|V_H|=5$ and, since $H$ is planar, $|E_H|\leq 3\cdot 5-6=9$. Therefore we can replace $H$ by a copy of \kfiveminus~and obtain another edge-extremal graph of $\planar(5,\nu)$ with larger number of connected components isomorphic to \kfiveminus, a contradiction.

 If $\nu_1=3$, then $|V_H|=7$. If $|E_H|\leq 13$, we can replace $H$ in $G$ by a copy of $\kfiveminus+K_{1,4}$ and reach a contradiction with our choice of $G$. We now show that $H$ cannot have strictly more than 13 edges. Assume for a contradiction that $|E_H|\geq 14$. Then $H$ must be 4-regular. To conclude the proof, we observe that there is no planar 4-regular graph on seven vertices. Let $\overline{H}$ be the complement of $H$. Then $\overline{H}$ is a 2-regular graph. Thus, $\overline{H}=C_7$ or $\overline{H}=C_4+C_3$. This contradicts the fact that $H$ is planar, since both the complement of $C_7$ and the one of $C_4+C_3$ are not planar graphs. 
\end{claimproof} 

If $G$ has two components that are isomorphic $K_{1,4}$, then we can replace these two components by a copy of \kfiveminus~and obtain a graph in $\planar(5,\nu)$ that has strictly more edges than $G$, a contradiction. Hence, $G$ has at most one component that is isomorphic to $K_{1,4}$. Since $\nu(K_{1,4})=1$, $\nu(\kfiveminus)=2$ and $\nu(G)\leq\nu_1-1$, we have that, if $\nu$ is odd, then $G$ is a disjoint union of $\frac{\nu_1 -1}{2}$ copies of \kfiveminus. If $\nu$ is even, then $G$ is a disjoint union of $\lfloor\frac{\nu-1}{2}\rfloor$ copies of \kfiveminus~and one copy of $K_{1,4}$. Hence $|E_G|= 9\lfloor\frac{\nu-1}{2}\rfloor+ 4 (\nu -1 \bmod 2)$ and $G$ is pivotal in $\planar(5,\nu)$. Since $G$ is an edge-extremal graph in $\planar(5,\nu)$, any other graph $G'\in \planar(5,\nu)$ is such that $|E_{G'}|\leq |E_{G}|$. This concludes the proof of Lemma~\ref{lemma:dis5}.
\end{proof}


\newcommand\fourstuff{\ensuremath{A_4}}
\newcommand\sevenstuff{\ensuremath{A_7}}
We now proceed to the case $d=6$. Let \fourstuff~and \sevenstuff~be the graphs depicted in Figure~\ref{fig:foursevenstuff}. Note that $\nu(\fourstuff)=4$, $|E_{\fourstuff}|=21$, $\nu(\sevenstuff)=7$ and $|E_{\sevenstuff}|=37$.

\begin{definition}
For each $\nu$, a graph $G$ 
	is \emph{pivotal in $\planar(6, \nu)$}
	if $G$ is isomorphic to  
\[
\begin{aligned}
	\left\lfloor\frac{\nu-1}{7}\right\rfloor\sevenstuff+ (\nu-1 \bmod 7)K_{1,5}, \quad & \text{if } (\nu-1 \bmod 7)\leq 3,\\
	\left\lfloor\frac{\nu-1}{7}\right\rfloor\sevenstuff+\fourstuff+((\nu-1 \bmod 7)-4)K_{1,5}, \quad & \text{if } (\nu-1 \bmod 7)\geq 4.
\end{aligned}
\]	
\end{definition}

\begin{figure}
	\centering
\begin{tikzpicture}
	\begin{scope}[xshift=-4cm, node distance=0.75cm]
		\fivestar
		\node[] at (0,-2.5) {$K_{1,5}$};
	\end{scope}
	\begin{scope}
		\node[circle, minimum size=1cm] (C1) at (0,0) {};
		\node[circle, minimum size=3.5cm] (C2) at (0,-0.375) {};
	
		\node[vtx] (1) at (0,0) {};
		\node[vtx] (2) at (C2.90) {};
		\node[vtx] (3) at (C1.180) {};
		\node[vtx] (4) at (C1.240) {};
		\node[vtx] (5) at (C1.300) {};
		\node[vtx] (6) at (C1.0) {};
	
		\node[vtx] (9) at (C2.210) {};
		\node[vtx] (10) at (C1.270) [below=0.25cm] {};
		\node[vtx] (12) at (C2.330) {};
	
		\foreach \i in {2,3,4,5,6} {
			\draw[edge] (1) -- (\i);
		}
		\foreach \i in {3,6,9,12} {
			\draw[edge] (2) -- (\i);
		}
		\foreach \i in {4,9} {
			\draw[edge] (3) -- (\i);
		}
		\foreach \i in {5,9,10} {
			\draw[edge] (4) -- (\i);
		}
		\foreach \i in {6,10,12} {
			\draw[edge] (5) -- (\i);
		}
		\foreach \i in {12} {
			\draw[edge] (6) -- (\i);
		}
	
		\foreach \i in {10,12} {
			\draw[edge] (9) -- (\i);
		}
		\foreach \i in {12} {
			\draw[edge] (10) -- (\i);
		}
		\node[] at (0,-2.5) {$\fourstuff$};
	\end{scope}
	\begin{scope}[xshift=6cm]
	\node[circle, minimum size=1cm] (C1) at (0,0) {};
	\node[circle, minimum size=2cm] (C2) at (0,0) {};
	\node[circle, minimum size=5cm] (C3) at (0,-0.375) {};
	
	\node[vtx] (1) at (0,0) {};
	\node[vtx] (2) at (C1.90) {};
	\node[vtx] (3) at (C1.180) {};
	\node[vtx] (4) at (C1.240) {};
	\node[vtx] (5) at (C1.300) {};
	\node[vtx] (6) at (C1.0) {};
	
	\node[vtx] (7) at (C2.70) {};
	\node[vtx] (8) at (C2.110) {};
	\node[vtx] (9) at (C2.230) {};
	\node[vtx] (10) at (C2.270) [above=0.125cm] {};
	\node[vtx] (11) at (C2.270) [below=0.25cm] {};
	\node[vtx] (12) at (C2.310) {};
	
	\node[vtx] (13) at (C3.90) {};
	\node[vtx] (14) at (C3.210) {};
	\node[vtx] (15) at (C3.330) {};
	
	\foreach \i in {2,3,4,5,6} {
		\draw[edge] (1) -- (\i);
	}
	\foreach \i in {3,6,7,8} {
		\draw[edge] (2) -- (\i);
	}
	\foreach \i in {4,8,9} {
		\draw[edge] (3) -- (\i);
	}
	\foreach \i in {5,9,10} {
		\draw[edge] (4) -- (\i);
	}
	\foreach \i in {6,10,12} {
		\draw[edge] (5) -- (\i);
	}
	\foreach \i in {7,12} {
		\draw[edge] (6) -- (\i);
	}
	\foreach \i in {8,13,15} {
		\draw[edge] (7) -- (\i);
	}
	\foreach \i in {13,14} {
		\draw[edge] (8) -- (\i);
	}
	\foreach \i in {10,11,14} {
		\draw[edge] (9) -- (\i);
	}
	\foreach \i in {11,12} {
		\draw[edge] (10) -- (\i);
	}
	\foreach \i in {12,14,15} {
		\draw[edge] (11) -- (\i);
	}
	\foreach \i in {15} {
		\draw[edge] (12) -- (\i);
	}
	\foreach \i in {14,15} {
		\draw[edge] (13) -- (\i);
	}
	\draw[edge] (14) -- (15);
	\node[] at (0,-2.5) {$\sevenstuff$};
	\end{scope}
\end{tikzpicture}
	\caption{Possible connected components of the pivotal graphs in $\planar(6,\nu)$.}
	\label{fig:foursevenstuff}
\end{figure}
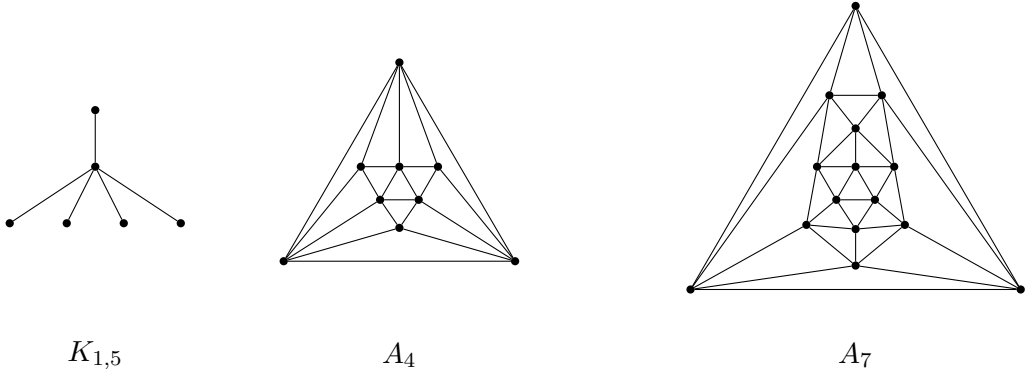

\begin{lemma}\label{lemma:dis6}
	If $G$ is a graph in $\planar(6, \nu)$,
	then, \[\card{E_G} \leq 5(\nu-1) + 2\cdot\left\lfloor \frac{\nu-1}{7}\right\rfloor + \ind[\nu-1 \bmod 7 \ge 4].\]
	Moreover, if $G$ is pivotal in $\planar(6,\nu)$, then equality holds.
\end{lemma}

\begin{proof}
It is straightforward to see that if $G$ is a pivotal graph in $\planar(6,\nu)$, then $\card{E_G} = 5(\nu-1) + 2\cdot\left\lfloor \frac{\nu-1}{7}\right\rfloor + \ind[\nu-1 \bmod 7 \ge 4]$.

	Let $G$ be an edge-extremal graph in $\planar(6, \nu)$ with maximum number of connected components that are isomorphic to stars, and subject to that, with maximum number of components isomorphic to \fourstuff~or \sevenstuff.
	
\begin{claim}
Every connected component of $G$ is isomorphic to \fourstuff, \sevenstuff~or to $K_{1,5}$.
\end{claim}

\begin{claimproof}
Suppose that $G$ has a component $H$ that is not isomorphic to $\fourstuff$, neither to $\sevenstuff$, nor to $K_{1,5}$. If $H$ is a star, then $H$ is isomorphic to $K_{1,a}$, with $a\leq 4$. So we can replace the connected component $H$ in $G$ by a copy of $K_{1,5}$ and obtain a graph that has degree at most five and the same matching number as $G$, but has strictly more edges than $G$. This is a contradiction with the fact that $G$ is edge-extremal in $\planar(6,\nu)$. Hence $H$ is not a star. Let $\nu_1$ the size of maximum matching in $H$. By Lemma~\ref{lem:factorc}, $H$ is factor-critical and therefore $\card{V_H} = 2\nu_1 + 1$. By Lemma~\ref{lem:planar:edges}, we have that $\card{E_H} \le 6\nu_1 - 3$.
	Therefore, if $\nu_1 \le 3$, we can replace $H$ in $G$ by $\nu_1 K_{1,5}$ without modifying the matching number of $G$, neither decreasing the number of edges in $G$. However, we obtain a graph that has strictly more connected components that are isomorphic to $K_{1,5}$, a contradiction with our initial choice of $G$. So we can assume $\nu_1\geq 4$ and consider the following four cases.
	
	If $\nu_1=4$, then $|V_H|=9$ and by Observation~\ref{lem:planar:edges}, $|E_H|\leq 21$. Hence we can replace $H$ in $G$ by a graph isomorphic to \fourstuff~without modifying the degree, the matching number, nor the number of edges of $G$. That is, we obtain another edge extremal graph in $\planar(6,\nu)$ that has more connected components isomorphic to \fourstuff~than $G$, a contradiction.
	
	If $\nu_1=5$, then $|V_H|=11$. By Theorem~\ref{thm:degreeseq}, the degree sequence of $H$ cannot be $5^{10}4$. Hence $|E_H|\leq 26$. In this case, we can replace $H$ in $G$ by a copy of $\fourstuff+K_{1,5}$ and obtain a graph in $\planar(6, \nu)$
	that has at least as many edges as $G$, but more components isomorphic to stars, a contradiction.

	If $\nu_1=6$, then $|V_H|=13$. By Theorem~\ref{thm:degreeseq}, the degree sequence of $H$ cannot be $5^{12}4$. Hence $|E_H|\leq 31$. So we can replace $H$ by a copy of $\fourstuff+2K_{1,5}$ and obtain a graph in $\planar(6, \nu)$ with at least as many edges as $G$, but more stars, a contradiction with the choice of $G$.
	
	Now suppose that $\nu_1 \ge 7$. Recall that $|V_H|=2\nu_1+1$. 
	The degree sequence yielding the largest possible number of edges in $H$ is $5^{2\nu_1}4$,
	in which case $\card{E_H} = 5\nu_1 + 2$.
	Let $H'$ be a pivotal graph in $\planar(6, \nu_1+1)$. Note that $|E_{H'}|= 5\nu_1 + 2\cdot\left\lfloor \frac{\nu_1}{7}\right\rfloor + \ind[\nu_1 \bmod 7 \ge 4]$. Thus, 
	\begin{align*}
		\card{E_{H'}} - \card{E_H} &= 5\nu_1 + 2\cdot\left\lfloor \frac{\nu_1}{7}\right\rfloor + \ind[\nu_1 \bmod 7 \ge 4] - (5 \nu_1 + 2) \\
			&= 2\cdot\left\lfloor \frac{\nu_1}{7}\right\rfloor + \ind[\nu_1 \bmod 7 \ge 4] - 2 \\
			&\ge 0.
	\end{align*}
	
	This implies that we can replace the component $H$ in $G$ by a graph that is isomorphic to $H'$ and obtain another graph in $\planar(6,\nu)$ that has at least as many edges as $G$, the same matching number as $G$, but it has more connected components isomorphic to $K_{1,5}$, \fourstuff~or \sevenstuff, a contradiction with our choice of $G$.
\end{claimproof}

If $G$ has two components isomorphic to \fourstuff, we can replace these two components by a copy of $\sevenstuff+K_{1,5}$, without changing the number of edges of $G$. Hence we can assume $G$ has at most one component isomorphic to \fourstuff. Note that, since $G$ is edge-extremal, $G$ has at most three components that are isomorphic to $K_{1,5}$. Indeed, if $G$ had four components isomorphic to $K_{1,5}$, we would be able to replace them by a copy of \fourstuff, and obtain a graph with more edges than $G$. Finally if $G$ contains $\fourstuff+3K_{1,5}$, we can replace this by a copy of \sevenstuff~without changing the number of edges of $G$. Hence, we can assume that:

\begin{enumerate}
\item[(i)] $G$ has at most one component isomorphic to \fourstuff~and at most three isomorphic to $K_{1,5}$;
\item[(ii)] if $G$ has a component isomorphic to \fourstuff, then $G$ has at most two components isomorphic to $K_{1,5}$.
\end{enumerate}

From this, we conclude that $G$ is pivotal in $\planar(6,\nu)$. Hence, $|E_G|= 5(\nu-1) + 2\cdot\left\lfloor \frac{\nu-1}{7}\right\rfloor + \ind[\nu-1 \bmod 7 \ge 4]$. Since $G$ is edge-extremal, any $G'\in \planar(6,\nu)$ is such that $|E_{G'}|\leq |E_G|$. This concludes the proof of Lemma~\ref{lemma:dis6}.	
\end{proof}


To conclude the proof of Theorem~\ref{thm:main}, we consider the case $d=7$.

\begin{lemma}\label{lemma:dis7}
If $G$ is a graph in $\planar(7,\nu)$, then $|E_G|\leq 6(\nu-1)$. Moreover, equality holds if $G$ is isomorphic to $(\nu-1)K_{1,6}$.
\end{lemma}

\begin{proof}
Let $G$ be an edge-extremal graph in $\planar(7,\nu)$ with maximum number of connected components that are isomorphic to stars. We will show that every component of $G$ is isomorphic to $K_{1,6}$. Suppose for a contradiction that $G$ has a connected component $H$ that is not isomorphic to $K_{1,6}$. Note that since $\Delta(H)\leq 6$, $H$ cannot be a star. Otherwise we would be able to replace $H$ by a graph isomorphic to $K_{1,6}$ and obtain a graph in $\planar(7,\nu)$ with more edges than $G$. Let $\nu_1$ be the size of a maximum matching in $H$. By Lemma~\ref{lem:factorc}, $H$ is factor critical and therefore $|V_H|=2\nu_1+1$. Since $H$ is a planar graph, $|E_H|\leq 3|V_H|-6$. This implies that $|E_H|\leq 3(2\nu_1+1)-6$, that is,  $|E_H|\leq 6\nu_1-3$. If we replace the component $H$ of $G$ by $\nu_1$ stars of degree $6$ we obtain a graph that still belongs to $\planar(7,\nu)$ but that has more edges than $G$, a contradiction. We conclude that $G$ is a disjoint union of stars and therefore $|E_G|= 6\cdot \nu(G)$. Since $G$ is edge-extremal in $\planar(7,\nu)$, any $G'\in \planar(7,\nu)$ has at most as many edges as $G$, which concludes the proof.
\end{proof}


Theorem~\ref{thm:main} now follows directly from Observation~\ref{obs:lessthan5} and Lemmas~\ref{lemma:dis5}, \ref{lemma:dis6} and \ref{lemma:dis7}.

\newcommand\fivestuff{\ensuremath{A_5}}
\newcommand\sixstuff{\ensuremath{A_6}}
\section{Conclusion}\label{sec:conclusion}

In this work, we determined the maximum number of edges that a planar graph
can have if its maximum degree and matching number are bounded. 
We also exhibited examples of graphs achieving this bound.

We point out that the edge-extremal graphs described in the proof of Theorem~\ref{thm:main} are not unique. For instance, the edge-extremal graphs described in Lemma~\ref{lemma:dis6} for the case $d=6$ have connected components that are isomorphic to \fourstuff, \sevenstuff~or $K_{1,5}$ (see Figure~\ref{fig:foursevenstuff}). However, the graphs \fivestuff~and \sixstuff~depicted in Figure~\ref{fig:pivotal6-alt} are also edge-extremal in $\planar(6,6)$ and $\planar(6,7)$, respectively. 

In view of this, it would be interesting to investigate how a connectivity constraint affects the number of edges of planar edge-extremal graphs. A more general question that remains open (also mentioned by Dibek et al.~\cite{DIBEK17} and Blair et al.~\cite{LATIN2020}) is: What is the maximum number of edges that a \emph{connected} graph can have as a function of its degree and matching number? The answer to this question might require a different approach, as Lemma~\ref{lem:factorc}, one of the main tools in all the results obtained so far, is no longer applicable.

\begin{figure}
	\centering
\begin{tikzpicture}
	\begin{scope}
		\node[circle, minimum size=1cm] (C1) at (0,0) {};
		\node[circle, minimum size=2cm] (C2) at (0,0) {};
		\node[circle, minimum size=5cm] (C3) at (0,-0.375) {};
		
		\node[vtx] (1) at (0,0) {};
		\node[vtx] (2) at (C1.210) {};
		\node[vtx] (3) at (C1.270) {};
		\node[vtx] (4) at (C1.330) {};
		\node[vtx] (5) at (C1.60) {};
		\node[vtx] (6) at (C1.120) {};
			
		\node[vtx] (7) at (C2.240) {};
		\node[vtx] (8) at (C2.300) {};
	
		\node[vtx] (9) at (C3.210) {};
		\node[vtx] (10) at (C3.330) {};
		\node[vtx] (11) at (C3.90) {};
	
		\foreach \i in {2,3,4,5,6} {
			\draw[edge] (1) -- (\i);
		}
		\foreach \i in {3,6,7,9} {
			\draw[edge] (2) -- (\i);
		}
		\foreach \i in {4,7,8} {
			\draw[edge] (3) -- (\i);
		}
		\foreach \i in {5,8,10} {
			\draw[edge] (4) -- (\i);
		}
		\foreach \i in {6,10,11} {
			\draw[edge] (5) -- (\i);
		}
		\foreach \i in {9,11} {
			\draw[edge] (6) -- (\i);
		}
		\foreach \i in {8,9} {
			\draw[edge] (7) -- (\i);
		}
		\foreach \i in {10} {
			\draw[edge] (8) -- (\i);
		}
		\foreach \i in {10,11} {
			\draw[edge] (9) -- (\i);
		}
		\foreach \i in {11} {
			\draw[edge] (10) -- (\i);
		}
	
		
		\node[] at (0,-2.5) {$\fivestuff$};
	\end{scope}
	\begin{scope}[xshift=7.5cm]
	\node[circle, minimum size=1cm] (C1) at (0,0) {};
	\node[circle, minimum size=2cm] (C2) at (0,0) {};
	\node[circle, minimum size=5cm] (C3) at (0,-0.375) {};
	
	\node[vtx] (1) at (0,0) {};
	\node[vtx] (2) at (C1.90) {};
	\node[vtx] (3) at (C1.180) {};
	\node[vtx] (4) at (C1.240) {};
	\node[vtx] (6) at (C1.0) {};
	
	\gettikzxy{(6)}{\xsix}{\ysix}
	\gettikzxy{(4)}{\xfour}{\yfour}
	\node[vtx] (5) at (\xsix,\yfour) {};
	
	\node[vtx] (7) at (C2.70) {};
	\node[vtx] (8) at (C2.110) {};
	\node[vtx] (9) at (C2.230) {};
	\node[vtx] (10) at (C2.270) {};
	
	\node[vtx] (11) at (C3.90) {};
	\node[vtx] (12) at (C3.210) {};
	\node[vtx] (13) at (C3.330) {};
	
	\foreach \i in {2,3,4,5,6} {
		\draw[edge] (1) -- (\i);
	}
	\foreach \i in {3,6,7,8} {
		\draw[edge] (2) -- (\i);
	}
	\foreach \i in {4,8,9} {
		\draw[edge] (3) -- (\i);
	}
	\foreach \i in {5,9,10} {
		\draw[edge] (4) -- (\i);
	}
	\foreach \i in {6,10,13} {
		\draw[edge] (5) -- (\i);
	}
	\foreach \i in {7,13} {
		\draw[edge] (6) -- (\i);
	}
	\foreach \i in {8,11} {
		\draw[edge] (7) -- (\i);
	}
	\foreach \i in {11,12} {
		\draw[edge] (8) -- (\i);
	}
	\foreach \i in {10,12} {
		\draw[edge] (9) -- (\i);
	}
	\foreach \i in {12,13} {
		\draw[edge] (10) -- (\i);
	}
	\foreach \i in {12,13} {
		\draw[edge] (11) -- (\i);
	}
	\foreach \i in {13} {
		\draw[edge] (12) -- (\i);
	}
	\node[] at (0,-2.5) {$\sixstuff$};
	
	
	\end{scope}
\end{tikzpicture}
	\caption{The graph \fivestuff~is edge-extremal in $\planar(6,6)$, while \sixstuff~is edge-extremal in $\planar(6,7)$.}
	\label{fig:pivotal6-alt}
\end{figure}
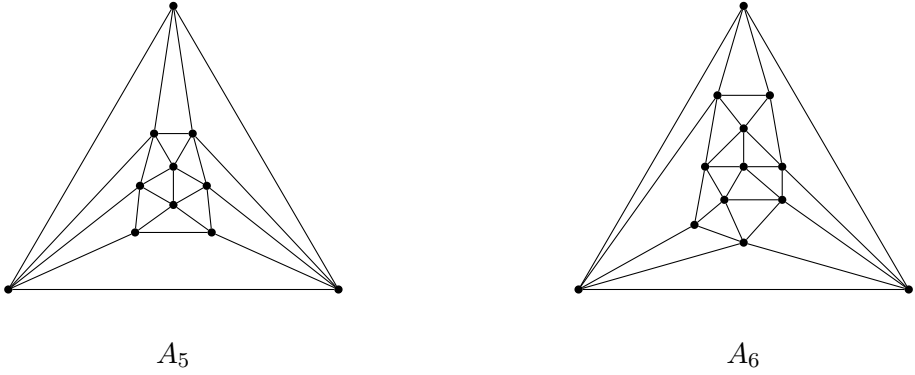

 \bibliographystyle{siam}

  \bibliography{ref}

\end{document}